\documentclass[10pt]{amsart}
\usepackage{amssymb}
\usepackage{amsthm,amsmath}
\usepackage{fancybox,epic,cite}

\title{Classifications of quasitrivial semigroups}

\author{Jimmy Devillet}
\address{Mathematics Research Unit, University of Luxembourg, Maison du Nombre, 6, avenue de la Fonte, L-4364 Esch-sur-Alzette, Luxembourg}
\email{jimmy.devillet[at]uni.lu}

\author{Jean-Luc Marichal}
\address{Mathematics Research Unit, University of Luxembourg, Maison du Nombre, 6, avenue de la Fonte, L-4364 Esch-sur-Alzette, Luxembourg}
\email{jean-luc.marichal[at]uni.lu}

\author{Bruno Teheux}\thanks{Corresponding author: Bruno Teheux is with the Mathematics Research Unit, University of Luxembourg, Maison du Nombre, 6, avenue de la Fonte, L-4364 Esch-sur-Alzette, Luxembourg.\\ Email: bruno.teheux[at]uni.lu}
\address{Mathematics Research Unit, University of Luxembourg, Maison du Nombre, 6, avenue de la Fonte, L-4364 Esch-sur-Alzette, Luxembourg}
\email{bruno.teheux[at]uni.lu}

\date{November 19, 2018}

\theoremstyle{plain}
\newtheorem{theorem}{Theorem}[section]
\newtheorem{lemma}[theorem]{Lemma}
\newtheorem{proposition}[theorem]{Proposition}
\newtheorem{corollary}[theorem]{Corollary}

\theoremstyle{definition}
\newtheorem{definition}[theorem]{Definition}

\theoremstyle{remark}

\newtheorem{remark}{Remark}

\newcommand{\Z}{\mathbb{Z}}

\begin{document}
\begin{abstract}
We investigate classifications of quasitrivial semigroups defined by certain equivalence relations. The subclass of quasitrivial semigroups that preserve a given total ordering is also investigated. In the special case of finite semigroups, we address and solve several related enumeration problems.
\end{abstract}

\keywords{Quasitrivial semigroup, classification, enumeration, group action.}

\subjclass[2010]{Primary 05A15, 20M14, 20M99; Secondary 39B52.}

\maketitle

\section{Introduction}

Let $X$ be an arbitrary nonempty set. We use the symbol $X_n$ if $X$ contains $n\geq 1$ elements, in which case we assume without loss of generality that $X_n=\{1,\ldots,n\}$. Recall that a binary operation $F\colon X^2\to X$ is said to be
\begin{itemize}
\item \emph{associative} if $F(F(x,y),z)=F(x,F(y,z))$ for all $x,y,z\in X$;
\item \emph{quasitrivial} if $F(x,y)\in\{x,y\}$ for all $x,y\in X$.
\end{itemize}

This paper focuses on the class of operations $F\colon X^2\to X$ that are both associative and quasitrivial. For such operations, the pair $(X,F)$ is then called a \emph{quasitrivial semigroup}. For recent references, see, e.g., \cite{Ack,CouDevMar2,DevKisMar}.

We let $\mathcal{F}$ be the class of associative and quasitrivial operations $F\colon X^2\to X$. We will often denote this class by $\mathcal{F}_n$ if $X=X_n$ for some integer $n\geq 1$. Although the class $\mathcal{F}$ has been completely characterized (see Theorem~\ref{thm:kimura} below), its structure can be investigated by classifying its elements into subclasses. The purpose of this paper is to define and analyze such classifications by considering natural equivalence relations. The case where $X$ is finite also raises the interesting problem of enumerating the corresponding equivalence classes.

The outline of this paper is as follows. In Section 2, we essentially recall a descriptive characterization of the class $\mathcal{F}$. In Section 3, we introduce and investigate classifications of the elements of $\mathcal{F}$ by defining three natural equivalence relations. One of these classifications is simply obtained by considering orbits (conjugacy classes) defined by letting the group of permutations on $X$ act on $\mathcal{F}$. We also focus on the finite case, where we enumerate the equivalence classes defined by each of these equivalence relations. In Section 4, we investigate the operations of $\mathcal{F}$ that are order-preserving for some total ordering on $X$. In particular, we characterize the above-mentioned orbits that contain at least one such order-preserving operation. We also elaborate on the finite case, where the enumeration problems give rise to new integer sequences. In Section 5, we examine further subclasses of $\mathcal{F}$ by considering additional properties: commutativity, anticommutativity, and bisymmetry.

Throughout this paper, the size of any finite set $S$ is denoted by $|S|$.

\section{Quasitrivial semigroups}

Recall that a \emph{weak ordering} on $X$ is a binary relation $\lesssim$ (or $\precsim$) on $X$ that is total and transitive. We write $x\sim y$ if $x\lesssim y$ and $y\lesssim x$. Also, we write $x<y$ if $x\lesssim y$ and $\neg(y\lesssim x)$. A \emph{total ordering} on $X$ is a weak ordering on $X$ that is antisymmetric. In this case, we have $x\sim y$ if and only if $x=y$. For this reason, we usually denote a total ordering by $\leq$ (or $\preceq$). In this paper we will sometimes endow $X_n$ with the usual total ordering relation $\leq_n$ defined by $1<_n\cdots <_n n$.

For any weak ordering $\lesssim$ on $X$, the relation $\sim$ is an equivalence relation on $X$ and the relation $<$ induces a total ordering on the quotient set ${X/{\sim}}$. Thus, a weak ordering on $X$ is nothing other than a totally ordered partition of $X$.

For any total ordering $\leq$ on $X$ and any weak ordering $\precsim$ on $X$, we say that $\leq$ \emph{extends} (or \emph{is subordinated to}) $\precsim$ if, for any $x,y\in X$, we have that $x\prec y$ implies $x<y$.

Given a weak ordering $\lesssim$ on $X$, the \emph{maximum} on $X$ for $\lesssim$ is the partial commutative binary operation $\max_{\lesssim}$ defined on $$X^2\setminus\{(x,y)\in X^2: x\sim y,~x\neq y\}$$ by $\max_{\lesssim}(x,y)=y$ whenever $x\lesssim y$. If $\lesssim$ reduces to a total ordering, then clearly the operation $\max_{\lesssim}$ is defined everywhere on $X^2$. Also, the \emph{projection operations} $\pi_1\colon X^2\to X$ and $\pi_2\colon X^2\to X$ are respectively defined by $\pi_1(x,y)=x$ and $\pi_2(x,y)=y$ for all $x,y\in X$.

The preimage of an element $z\in X$ under an operation $F\colon X^2\to X$ is denoted by $F^{-1}[z]$. When $X=X_n$ for some integer $n\geq 1$, we also define the \emph{preimage sequence of $F$} as the nondecreasing $n$-element sequence of the numbers $|F^{-1}[z]|$, $z\in X_n$. We denote this sequence by $|F^{-1}|$.

The following theorem provides a descriptive characterization of the class $\mathcal{F}$. As recently observed in \cite{Ack}, this characterization can be easily derived from two results from the 1950s on idempotent semigroups. A recent discussion and a direct elementary proof can be found in \cite{CouDevMar2}.

\begin{theorem}\label{thm:kimura}
We have $F\in\mathcal{F}$ if and only if there exists a weak ordering $\precsim$ on $X$ such that
$$
F|_{A\times B} ~=~
\begin{cases}
\pi_1|_{A\times B}\hspace{1.5ex}\text{or}\hspace{1.5ex}\pi_2|_{A\times B}, & \text{if $A=B$},\\
\max_{\precsim}|_{A\times B}, & \text{if $A\neq B$},
\end{cases}
\qquad \forall A,B\in {X/{\sim}}.
$$
\end{theorem}

It is not difficult to see that the weak ordering $\precsim$ mentioned in Theorem~\ref{thm:kimura} is unique. The following proposition provides a way to construct it.

\begin{proposition}[{see \cite{CouDevMar2}}]\label{prop:woKI}
The weak ordering $\precsim$ mentioned in Theorem~\ref{thm:kimura} is uniquely determined from $F$ and is defined by
\begin{equation}\label{eq:1of1}
x\precsim y\quad\Leftrightarrow\quad F(x,y)=y \hspace{1.5ex}\text{or}\hspace{1.5ex} F(y,x)=y,\qquad x,y\in X.
\end{equation}
If $X=X_n$ for some integer $n\geq 1$, then we also have the equivalence
\begin{equation}\label{eq:of1}
x\precsim y\quad\Leftrightarrow\quad |F^{-1}[x]|\,\leq\, |F^{-1}[y]|,\qquad x,y\in X.
\end{equation}
\end{proposition}

\begin{remark}
Condition \eqref{eq:of1} was equivalently stated in \cite{CouDevMar2} in terms of $F$-degrees, where the $F$-degree of an element $z\in X_n$ is the natural integer $\deg_F(z)=|F^{-1}[z]|-1$.
\end{remark}

In this paper, the weak ordering $\precsim$ on $X$ defined by \eqref{eq:1of1} from any $F\in\mathcal{F}$ will henceforth be denoted by $\precsim_F$.

The following immediate corollary provides an alternative characterization of the class $\mathcal{F}$ that does not make use of the concept of weak ordering. Here the axiom of choice is required to choose total orderings in a collection of sets.

\begin{corollary}\label{cor:kimura}
Assume the axiom of choice. We have $F\in\mathcal{F}$ if and only if there exists a total ordering $\preceq$ on $X$, a partition of $X$ into nonempty $\preceq$-convex sets $\{C_i:i\in I\}$, and a map $\varepsilon\colon I\to\{1,2\}$ such that
$$
F(x,y) ~=~
\begin{cases}
\pi_{\varepsilon(i)}(x,y), & \text{if $\exists\, i\in I$ such that $x,y\in C_i$},\\
\max_{\preceq}(x,y), & \text{otherwise},
\end{cases}
\qquad \forall x,y\in X.
$$
\end{corollary}

An operation $F\colon X^2\to X$ of the form given in Corollary~\ref{cor:kimura} is called an \emph{ordinal sum} \cite{Lan80} of projections on the totally ordered set $(X,\preceq)$. Such an operation is illustrated in Figure~\ref{fig:osp}. Combining Theorem~\ref{thm:kimura} and Corollary~\ref{cor:kimura}, we can easily see that any $F\in\mathcal{F}$ is an ordinal sum  of projections on $(X,\preceq)$ if and only if $\preceq$ extends $\precsim_F$ (i.e, we have $a\prec b$ whenever $a\prec_F b$).

\setlength{\unitlength}{4.8ex}
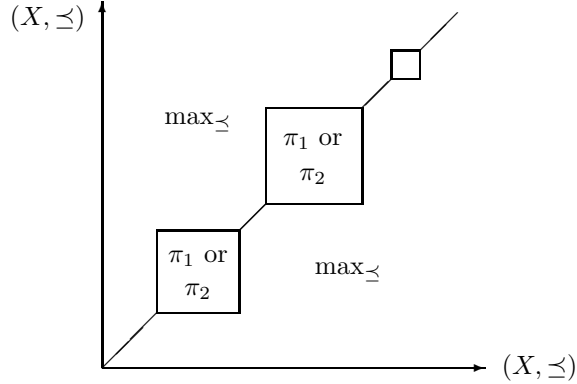
\begin{figure}[htbp]
\begin{center}
\begin{picture}(6.5,6.5)
\put(0,0){\vector(1,0){7}}\put(0,0){\vector(0,1){6.7}}
\put(8,0){\makebox(0,0){$(X,\preceq)$}}\put(-1,6.4){\makebox(0,0){$(X,\preceq)$}}
\put(1,1){\framebox(1.5,1.5){}}
\put(3,3){\framebox(1.75,1.75){}}\put(5.3,5.3){\framebox(0.5,0.5){}}
\put(1.75,2){\makebox(0,0){$\pi_1$ or}}\put(1.75,1.4){\makebox(0,0){$\pi_2$}}
\put(3.85,4.15){\makebox(0,0){$\pi_1$ or}}\put(3.85,3.5){\makebox(0,0){$\pi_2$}}
\put(1.75,4.5){\makebox(0,0){$\max_{\preceq}$}}\put(4.5,1.75){\makebox(0,0){$\max_{\preceq}$}}
\drawline[1](0,0)(1,1)\drawline[1](2.5,2.5)(3,3)\drawline[1](4.75,4.75)(5.3,5.3)\drawline[1](5.8,5.8)(6.5,6.5)
\end{picture}
\caption{An ordinal sum of projections}
\label{fig:osp}
\end{center}
\end{figure}

Let us now assume that $X=X_n$ for some integer $n\geq 1$. Define the \emph{contour plot} of any operation $F\colon X_n^2\to X_n$ by the undirected graph $\mathcal{C}_F=(X_n^2,E)$, where
$$
E ~=~ \{\{(x,y),(u,v)\}:(x,y)\neq (u,v)~\text{and}~F(x,y)=F(u,v)\}.
$$
We observe that, for any $z\in X_n$ such that $F^{-1}[z]\neq\varnothing$, the subgraph of $\mathcal{C}_F$ induced by $F^{-1}[z]$ is a complete connected component of $\mathcal{C}_F$. It is also clear that $\mathcal{C}_F$ has exactly $|F(X^2_n)|$ connected components. In particular, $\mathcal{C}_F$ has $n$ connected components for every $F\in\mathcal{F}_n$.

We can always represent the contour plot of any operation $F\colon X_n^2\to X_n$ by fixing a total ordering on $X_n$. For instance, using the usual total ordering $\leq_6$ on $X_6$, in Figure~\ref{fig:2ab} (left) we represent the contour plot of an operation $F\colon X_6^2\to X_6$. To simplify the representation of the connected components, we omit edges that can be obtained by transitivity. The weak ordering $\precsim$ on $X_6$ obtained from \eqref{eq:of1} is such that $3\sim 4\prec 2\prec 1\sim 5\sim 6$. In Figure~\ref{fig:2ab} (right) we represent the contour plot of $F$ by using a total ordering $\leq$ on $X_6$ that extends $\precsim$. We then obtain an ordinal sum of projections on $\leq$, which finally shows that $F\in\mathcal{F}_6$ and that ${\precsim_F}={\precsim}$.

\setlength{\unitlength}{3.5ex}
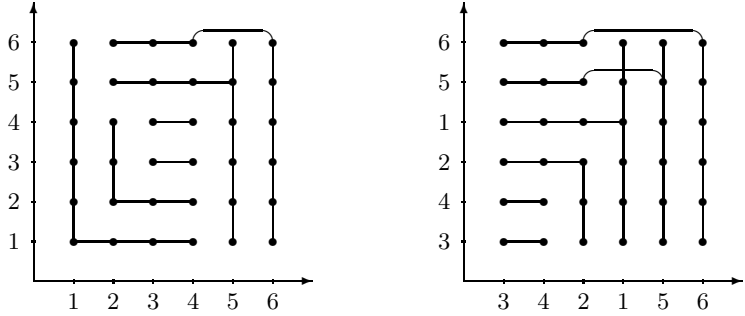
\begin{figure}[htbp]
\begin{center}
\begin{small}
\null\hspace{0.03\textwidth}
\begin{picture}(8,8)
\put(0.5,0.5){\vector(1,0){7}}\put(0.5,0.5){\vector(0,1){7}}
\multiput(1.5,0.45)(1,0){6}{\line(0,1){0.1}}%
\multiput(0.45,1.5)(0,1){6}{\line(1,0){0.1}}%
\put(1.5,0){\makebox(0,0){$1$}}\put(2.5,0){\makebox(0,0){$2$}}\put(3.5,0){\makebox(0,0){$3$}}
\put(4.5,0){\makebox(0,0){$4$}}\put(5.5,0){\makebox(0,0){$5$}}\put(6.5,0){\makebox(0,0){$6$}}
\put(0,1.5){\makebox(0,0){$1$}}\put(0,2.5){\makebox(0,0){$2$}}\put(0,3.5){\makebox(0,0){$3$}}
\put(0,4.5){\makebox(0,0){$4$}}\put(0,5.5){\makebox(0,0){$5$}}\put(0,6.5){\makebox(0,0){$6$}}
\multiput(1.5,1.5)(0,1){6}{\multiput(0,0)(1,0){6}{\circle*{0.2}}}
\drawline[1](1.5,6.5)(1.5,1.5)(4.5,1.5)\drawline[1](2.5,4.5)(2.5,2.5)(4.5,2.5)\drawline[1](3.5,3.5)(4.5,3.5)
\drawline[1](3.5,4.5)(4.5,4.5)\drawline[1](5.5,1.5)(5.5,6.5)\drawline[1](6.5,1.5)(6.5,6.5)
\drawline[1](2.5,5.5)(5.5,5.5)\drawline[1](2.5,6.5)(4.5,6.5)\put(5.5,6.5){\oval(2,0.6)[t]}
\end{picture}
\hspace{0.1\textwidth}
\begin{picture}(8,8)
\put(0.5,0.5){\vector(1,0){7}}\put(0.5,0.5){\vector(0,1){7}}
\multiput(1.5,0.45)(1,0){6}{\line(0,1){0.1}}%
\multiput(0.45,1.5)(0,1){6}{\line(1,0){0.1}}%
\put(1.5,0){\makebox(0,0){$3$}}\put(2.5,0){\makebox(0,0){$4$}}\put(3.5,0){\makebox(0,0){$2$}}
\put(4.5,0){\makebox(0,0){$1$}}\put(5.5,0){\makebox(0,0){$5$}}\put(6.5,0){\makebox(0,0){$6$}}
\put(0,1.5){\makebox(0,0){$3$}}\put(0,2.5){\makebox(0,0){$4$}}\put(0,3.5){\makebox(0,0){$2$}}
\put(0,4.5){\makebox(0,0){$1$}}\put(0,5.5){\makebox(0,0){$5$}}\put(0,6.5){\makebox(0,0){$6$}}
\multiput(1.5,1.5)(0,1){6}{\multiput(0,0)(1,0){6}{\circle*{0.2}}}
\drawline[1](1.5,3.5)(3.5,3.5)(3.5,1.5)\drawline[1](1.5,1.5)(2.5,1.5)\drawline[1](1.5,2.5)(2.5,2.5)
\drawline[1](4.5,1.5)(4.5,6.5)\drawline[1](5.5,1.5)(5.5,6.5)\drawline[1](6.5,1.5)(6.5,6.5)
\drawline[1](1.5,4.5)(4.5,4.5)\drawline[1](1.5,5.5)(3.5,5.5)\drawline[1](1.5,6.5)(3.5,6.5)
\put(4.5,5.5){\oval(2,0.6)[t]}\put(5,6.5){\oval(3,0.6)[t]}
\end{picture}
\end{small}
\caption{An operation $F\in\mathcal{F}_6$ (left) and its ordinal sum representation (right)}
\label{fig:2ab}
\end{center}
\end{figure}

This example clearly illustrates the following simple test to check whether a given operation $F\colon X_n^2\to X_n$ is associative and quasitrivial. First, use condition \eqref{eq:of1} to construct the unique weak ordering $\precsim$ on $X_n$ from the preimage sequence $|F^{-1}|$. Then, extend this weak ordering to a total ordering $\leq$ on $X_n$ and check if $F$ is an ordinal sum of projections on $\leq$. This test can be easily performed in $O(n^2)$ time.

Let us now present a result that will be useful as we continue. Recall first that two undirected graphs $G=(V,E)$ and $G'=(V',E')$ are said to be \emph{isomorphic}, and we write $G\simeq G'$, if there exists a bijection $\phi\colon V'\to V$ such that
$$
\{x,y\}\in E'\quad\Leftrightarrow\quad\{\phi(x),\phi(y)\}\in E,\qquad x,y\in V'.
$$
The bijection $\phi$ is then called an \emph{isomorphism} from $G'$ to $G$. It is called an \emph{automorphism} of $G$ if $G'=G$.

\begin{proposition}\label{prop:FG1C}
For any two operations $F\colon X_n^2\to X_n$ and $G\colon X_n^2\to X_n$, we have $|F^{-1}|=|G^{-1}|$ if and only if $\mathcal{C}_F\simeq\mathcal{C}_G$.
\end{proposition}

\begin{proof}
(Sufficiency) Trivial.

(Necessity) Recall that the \emph{order} of a graph is simply the number of its vertices. Thus, by definition, $|F^{-1}|$ is the nondecreasing $n$-element sequence of the orders of the connected components of $\mathcal{C}_F$. If $|F^{-1}|=|G^{-1}|$, then it is not difficult to construct a bijection $\phi\colon X^2_n\to X^2_n$ that maps a connected component of $\mathcal{C}_F$ to a connected component of $\mathcal{C}_G$ of the same order. Since all these connected components are complete subgraphs, we obtain that $\mathcal{C}_F\simeq\mathcal{C}_G$.
\end{proof}

\section{Classifications of quasitrivial semigroups}

It is a fact that the class $\mathcal{F}$ is generally very huge. In the finite case, the size of the class $\mathcal{F}_n$ (a sequence recorded in the OEIS as Sloane's A292932; see \cite{Slo}) becomes very large as $n$ grows (see \cite[Theorem 4.1]{CouDevMar2}).\footnote{In fact, we have $|\mathcal{F}_n| \sim \frac{1}{2\lambda+1}{\,}n!{\,}\lambda^{n+2}$ as $n\to\infty$, where $\lambda ~(\approx 1.71)$ is the inverse of the unique positive zero of the real function $x\mapsto x+3-2e^x$.} It is then natural to classify the elements of $\mathcal{F}$ by considering relevant equivalence relations on this class. Before introducing such relations, let us recall some basic definitions.

Recall first that two weak orderings $\lesssim$ and $\lesssim'$ on $X$ are said to be \emph{isomorphic}, and we write ${\lesssim}\simeq{\lesssim'}$, if there exists a bijection $\phi\colon X\to X$ such that
$$
x\lesssim' y\quad\Leftrightarrow\quad\phi(x)\lesssim\phi(y),\qquad x,y\in X.
$$
The bijection $\phi$ is then called an \emph{isomorphism} from $(X,\lesssim')$ to $(X,\lesssim)$. It is called an \emph{automorphism} of $(X,\lesssim)$ if ${\lesssim'}={\lesssim}$.

Let $\mathfrak{S}$ be the group of permutations on $X$. We will often denote this group by $\mathfrak{S}_n$ if $X=X_n$ for some integer $n\geq 1$. For any operation $F\colon X^2\to X$ and any permutation $\sigma\in\mathfrak{S}$, the \emph{$\sigma$-conjugate of $F$} is the operation $F_{\sigma}\colon X^2\to X$ defined by
$$
F_{\sigma}(x,y) ~=~ \sigma(F(\sigma^{-1}(x),\sigma^{-1}(y))),\qquad x,y\in X.
$$
A \emph{conjugate of $F$} is a $\sigma$-conjugate of $F$ for some $\sigma\in\mathfrak{S}$.

Clearly, the map $\psi\colon\mathfrak{S}\times\mathcal{F}\to\mathcal{F}$ defined by $\psi(\sigma,F)=F_{\sigma}$ is a group action of $\mathfrak{S}$ on $\mathcal{F}$. We then can define
\begin{itemize}
\item the \emph{orbit} of $F\in\mathcal{F}$ by $\mathrm{orb}(F)=\{F_{\sigma}:\sigma\in\mathfrak{S}\}$,
\item the \emph{stabilizer subgroup} of $\mathfrak{S}$ for $F\in\mathcal{F}$ by $\mathrm{stab}(F)=\{\sigma\in\mathfrak{S}: F_{\sigma}=F\}$.
\end{itemize}

We can readily see that, for any $\sigma\in\mathfrak{S}$, we have $F\in\mathcal{F}$ if and only if $F_{\sigma}\in\mathcal{F}$. Moreover, using \eqref{eq:1of1} we see that, for any $\sigma\in\mathfrak{S}$ and any $F\in\mathcal{F}$ we have
\begin{equation}\label{eq:isoWO}
x\precsim_F y\quad\Leftrightarrow\quad\sigma(x)\precsim_{F_{\sigma}}\sigma(y){\,},\qquad x,y\in X.
\end{equation}

Now, let $q$ be the identity relation on $\mathcal{F}$. We also introduce relations $p$, $s$, $r$ on $\mathcal{F}$ as follows. For any $F,G\in\mathcal{F}$, we write
\begin{eqnarray*}
F{\,}p{\,}G, &\text{if}& {\precsim_F}={\precsim_G}\\
F{\,}s{\,}G, &\text{if}& {\precsim_F}\simeq{\precsim_G}\\
F{\,}r{\,}G, &\text{if}& G\in\mathrm{orb}(F){\,}.
\end{eqnarray*}

It is clear that each of the relations above is an equivalence relation and hence it partitions $\mathcal{F}$ into equivalence classes. Moreover, we clearly have that $q\subseteq p\subseteq s$. Using \eqref{eq:isoWO}, we also see that $q\subseteq r\subseteq s$. Furthermore, we observe that $p$ and $r$ are not comparable in general. Indeed, if $F=\pi_1$ and $G=\pi_2$ on $X$, then we have $F{\,}p{\,}G$ and $\neg(F{\,}r{\,}G)$. Similarly, if $F=\max_{\preceq}$ and $G=\min_{\preceq}$ for some total ordering $\preceq$ on $X$, then we have $F{\,}r{\,}G$ and $\neg(F{\,}p{\,}G)$.

The following proposition provides further properties of the relations introduced above. Let us first investigate the conjunction of relations $p$ and $r$.

We observe that, given an operation $F\in\mathcal{F}$, any permutation $\sigma\in\mathfrak{S}$ for which $\sigma(x)\sim_F x$ for all $x\in X$ is an automorphism of $(X,\precsim_F)$. We say that such an automorphism is \emph{trivial}. It is easy to prove by induction that all automorphisms of $(X,\precsim_F)$ are trivial whenever $X$ is finite.

\begin{lemma}\label{lemma:cy64}
Let $F\in\mathcal{F}$ and $\sigma\in\mathfrak{S}$. Consider the following four assertions.
\begin{itemize}
\item[(i)] $F{\,}p{\,}F_{\sigma}$.
\item[(ii)] $F=F_{\sigma}$.
\item[(iii)] $\sigma(x)\sim_F x$ for every $x\in X$.
\item[(iv)] $\sigma$ is an automorphism of $(X,\precsim_F)$.
\end{itemize}
Then we have (iii) $\Rightarrow$ (ii) $\Rightarrow$ (i) $\Leftrightarrow$ (iv). The implication (iv) $\Rightarrow$ (iii) holds if and only if all automorphisms of $(X,\precsim_F)$ are trivial. The latter condition holds for instance if $X$ is finite.
\end{lemma}

\begin{proof}
(i) $\Leftrightarrow$ (iv). Straightforward (simply use \eqref{eq:isoWO}).

(iii) $\Rightarrow$ (iv). Trivial.

(iii) $\Rightarrow$ (ii). Let $x,y\in X$. Suppose first that $x\sim_F y$. By \eqref{eq:isoWO} and conditions (iii) and (iv), we have that $x\sim_F\sigma^{-1}(x)\sim_F\sigma^{-1}(y)\sim_F y$. Hence, by Theorem~\ref{thm:kimura} there exists $i\in\{1,2\}$ such that
\begin{eqnarray*}
F(\sigma^{-1}(x),\sigma^{-1}(y)) &=& \pi_i(\sigma^{-1}(x),\sigma^{-1}(y)) ~=~ \sigma^{-1}(\pi_i(x,y))\\
&=& \sigma^{-1}(F(x,y)),
\end{eqnarray*}
that is, $F_{\sigma}(x,y)=F(x,y)$. We proceed similarly if $x\prec_F y$ or $y\prec_F x$.

(ii) $\Rightarrow$ (i) Trivial.

The last part of the lemma is trivial.
\end{proof}

\begin{proposition}\label{prop:cy64}
We have $p\vee r=p\circ r=s$. If $X$ is finite, we also have $p\wedge r=q$.
\end{proposition}

\begin{proof}
Let us prove the first two identities. Since $p\circ r\subseteq p\vee r$, it is enough to show that $s\subseteq p\circ r$. Let $F,G\in\mathcal{F}$ such that $F{\,}s{\,}G$. That is, there exists $\sigma\in\mathfrak{S}$ such that
$$
x\precsim_G y\quad\Leftrightarrow\quad\sigma(x)\precsim_F\sigma(y),\qquad x,y\in X.
$$
Using \eqref{eq:isoWO}, we then see that
$$
x\precsim_F y\quad\Leftrightarrow\quad\sigma^{-1}(x)\precsim_G\sigma^{-1}(y)\quad\Leftrightarrow\quad x\precsim_{G_{\sigma}} y,\qquad x,y\in X,
$$
which means that ${\precsim_F}={\precsim_{G_{\sigma}}}$. Therefore we have $F{\,}p{\,}G_{\sigma}{\,}r{\,}G$, from which we derive that $s\subseteq p\circ r$.

To prove the last identity, we only need to show that $p\wedge r\subseteq q$. Let $F,G\in\mathcal{F}$ such that $F{\,}p{\,}G$ and $F{\,}r{\,}G$. By Lemma~\ref{lemma:cy64}, we have $F=G$, that is, $F{\,}q{\,}G$.
\end{proof}

\begin{remark}
We can easily construct operations $F\in\mathcal{F}$ for which $(X,\precsim_F)$ has nontrivial automorphisms. Consider for instance the operation $F=\max_{\leq}$ on $X=\Z$, where $\leq$ is the usual ordering on $\Z$, and take $\sigma(x)=x+1$. Then, we have $F\in\mathcal{F}$ and $\sigma\in\mathfrak{S}$. Also, conditions (i), (ii), and (iv) of Lemma~\ref{lemma:cy64} hold but condition (iii) fails to hold. Now, define the operation $F\colon\Z^2\to\Z$ by
$$
F(x,y) ~=~
\begin{cases}
x, & \text{if $(x,y)\in\{0,1\}^2$},\\
y, & \text{if $(x,y)\in\bigcup_{m\in\Z\setminus\{0\}}\{2m,2m+1\}^2$},\\
\max_{\leq}(x,y), & \text{otherwise},
\end{cases}
$$
where $\leq$ is the usual ordering on $\Z$. Take also $\sigma(x)=x-2$. Then, again we have $F\in\mathcal{F}$ and $\sigma\in\mathfrak{S}$. Also, conditions (i) and (iv) of Lemma~\ref{lemma:cy64} hold but condition (iii) fails to hold. Moreover, we have $F(0,1)=0\neq 1=F_{\sigma}(0,1)$, which shows that condition (ii) fails to hold, and hence that $p\wedge r\neq q$.
\end{remark}

\begin{proposition}\label{prop:s7adf5}
For any $\sigma\in\mathfrak{S}$, the map $\tilde{\sigma}\colon\mathcal{F}/p\to \mathcal{F}/p$ defined by $\tilde{\sigma}(F/p)=F_{\sigma}/p$ is a (well-defined) permutation.
\end{proposition}

\begin{proof}
Let $\sigma\in\mathfrak{S}$. For any $F,G\in\mathcal{F}$, by \eqref{eq:isoWO} we have $F{\,}p{\,}G$ if and only if $F_{\sigma}{\,}p{\,}G_{\sigma}$, which shows that $\tilde{\sigma}$ is well defined and injective. Now, for any $F\in\mathcal{F}$, we have $\tilde{\sigma}(F_{\sigma^{-1}}/p)=F/p$, which shows that $\tilde{\sigma}$ is also surjective.
\end{proof}

In the rest of this section we restrict ourselves to the finite case when $X=X_n$ for some integer $n\geq 1$. This assumption will enable us to enumerate the equivalence classes for each of the equivalence relations introduced above. For any integer $n\geq 1$, define $p(n)=|\mathcal{F}_n/p|$, $q(n)=|\mathcal{F}_n/q|$, $r(n)=|\mathcal{F}_n/r|$, and $s(n)=|\mathcal{F}_n/s|$. The first few values of these sequences are given in Table~\ref{tab:pqrs}.

\begin{table}[htbp]
$$
\begin{array}{|c|rrrr|}
\hline n & p(n) & q(n) & r(n) & s(n) \\
\hline 1 & 1 & 1 & 1 & 1 \\
2 & 3 & 4 & 3 & 2 \\
3 & 13 & 20 & 7 & 4 \\
4 & 75 & 138 & 17 & 8 \\
5 & 541 & 1{\,}182 & 41 & 16 \\
6 & 4{\,}683 & 12{\,}166 & 99 & 32 \\
\hline
\mathrm{OEIS}^{\mathstrut} & \mathrm{A000670} & \mathrm{A292932} & \mathrm{A001333} & \mathrm{A011782} \\
\hline
\end{array}
$$
\caption{First few values of $p(n)$, $q(n)$, $r(n)$, and $s(n)$}
\label{tab:pqrs}
\end{table}

By definition, $p(n)$ is the number of weak orderings on $X_n$, or equivalently, the number of totally ordered partitions of $X_n$ (Sloane's $\mathrm{A000670}$). Also, we clearly have $q(n)=|\mathcal{F}_n|$ and this number was recently computed in \cite{CouDevMar2} (Sloane's $\mathrm{A292932}$). Let us now investigate the numbers $r(n)$ and $s(n)$.

For any $F\in\mathcal{F}_n$, we set $k=|X_n/{\sim}_F|$ and let $C_1,\ldots,C_k$ denote the elements of ${X_n/{\sim_F}}$ ordered by the relation induced by $\precsim_F$, that is, $C_1\prec_F\cdots\prec_F C_k$ (where $C_i\prec_F C_j$ means that we have $x\prec_F y$ for all $x\in C_i$ and all $y\in C_j$). Also, we set $n_i=|C_i|$ for $i=1,\ldots,k$. We then define the \emph{signature} of $F$ as the $k$-tuple $(n_1,\ldots,n_k)$ and we denote it by $\mathrm{sgn}(F)$.

It is clear that the number of possible signatures in $\mathcal{F}_n$ is precisely the number of totally ordered partitions of a set of $n$ unlabeled items (Sloane's $\mathrm{A011782}$), that is,
$$
\sum_{k=1}^n ~\sum_{\textstyle{n_1,\ldots,n_k\geq 1\atop n_1+\cdots +n_k=n}}1~=~2^{n-1}{\,}.
$$
It follows that this number is also the number of weak orderings on $X_n$ that are defined up to an isomorphism. Thus, we have $s(n)=2^{n-1}$ for all $n\geq 1$.

We actually have the following more general result.

\begin{proposition}\label{prop:444}
For any $F,G\in\mathcal{F}_n$, the following assertions are equivalent.
\begin{enumerate}
\item[(i)] $F{\,}s{\,}G$.
\item[(ii)] $\mathcal{C}_F\simeq\mathcal{C}_G$.
\item[(iii)] $|F^{-1}|=|G^{-1}|$.
\item[(iv)] $\mathrm{sgn}(F)=\mathrm{sgn}(G)$.
\end{enumerate}
\end{proposition}

\begin{proof}
(i) $\Leftrightarrow$ (iii). Clearly, $|F^{-1}|=|G^{-1}|$ holds if and only if there exists $\sigma\in\mathfrak{S}_n$ such that $|F^{-1}[x]|=|G^{-1}[\sigma(x)]|$ for every $x\in X_n$. The claimed equivalence then immediately follows from condition \eqref{eq:of1}.

(ii) $\Leftrightarrow$ (iii). This result is a special case of Proposition~\ref{prop:FG1C}.

(i) $\Leftrightarrow$ (iv). Straightforward.
\end{proof}

The following proposition provides explicit expressions (bounded above by $n!$) for $|\mathrm{stab}(F)|$ and $|\mathrm{orb}(F)|$ for any $F\in\mathcal{F}_n$. In particular, it shows that $|\mathrm{orb}(F)|$ is precisely the number of ways to partition $X_n$ into $k$ subsets of sizes $n_1,\ldots,n_k$.\footnote{Recall that $k=|X_n/{\sim}_F|$ and that $n_i=|C_i|$ for $i=1,\ldots,k$.}

\begin{proposition}\label{prop:staborb}
For any $F\in\mathcal{F}_n$, we have
$$
|\mathrm{stab}(F)| ~=~ \prod_{i=1}^k n_i!\quad\text{and}\quad |\mathrm{orb}(F)| ~=~ {n\choose n_1,\ldots,n_k}.
$$
\end{proposition}

\begin{proof}
The formula for $|\mathrm{stab}(F)|$ immediately follows from Lemma~\ref{lemma:cy64}. By the classical orbit-stabilizer theorem, we have $|\mathrm{orb}(F)|\times|\mathrm{stab}(F)|=|\mathfrak{S}_n|$ for every $F\in\mathcal{F}_n$. This immediately proves the formula for $|\mathrm{orb}(F)|$.
\end{proof}

Recall that $r(n)$ is the number of orbits in $\mathcal{F}_n$ under the action of $\mathfrak{S}_n$. Burnside's lemma then immediately provides the formula
$$
r(n) ~=~ \frac{1}{n!}{\,}\sum_{\sigma\in\mathfrak{S}_n}|\mathcal{F}_n^{\sigma}|{\,},\qquad n\geq 1,
$$
where $\mathcal{F}_n^{\sigma}=\{F\in\mathcal{F}_n: F_{\sigma}=F\}$.

The following proposition provides much simpler explicit expressions for $r(n)$ and shows that the corresponding sequence is known as Sloane's $\mathrm{A001333}$, where we have set $r(0)=1$.

\begin{lemma}\label{lemma:leqn}
For every $F\in\mathcal{F}_n$, there exists a unique $\sigma\in\mathfrak{S}_n$ such that $F_{\sigma}$ is an ordinal sum of projections on $\leq_n$.
\end{lemma}

\begin{proof}
(Uniqueness) Let $\sigma,\mu\in\mathfrak{S}_n$ be such that $F_{\sigma}$ and $F_{\mu}$ are ordinal sums of projections on $\leq_n$. Since $F_{\sigma}$ and $F_{\mu}$ have the same signature, it follows that $F_{\sigma}{\,}p{\,}F_{\mu}$. By Lemma~\ref{lemma:cy64}, we then have $F_{\sigma}=F_{\mu}$, or equivalently, $\sigma=\mu$.

(Existence) By Corollary \ref{cor:kimura}, $F$ is an ordinal sum of projections on some total ordering $\leq$ that extends $\precsim_F$. Take $\sigma\in\mathfrak{S}_n$ such that
$$
\sigma(x)\leq_n\sigma(y)\quad\Leftrightarrow\quad x\leq y,\qquad x,y\in X_n.
$$
We then immediately see that $\leq_n$ extends $\precsim_{F_{\sigma}}$. Hence $F_{\sigma}$ is an ordinal sum of projections on $\leq_n$.
\end{proof}

\begin{proposition}\label{prop:d6fssf}
The sequence $(r(n))_{n\geq 0}$ satisfies the linear recurrence equation
$$
r(n+2) ~=~ 2r(n+1)+r(n),
$$
with $r(0)=1$ and $r(1)=1$. Its (ordinary) generating function is $R(z)=(1-z)/(1-2z-z^2)$. Moreover we have
$$
r(n) ~=~ \textstyle{\frac{1}{2}(1+\sqrt{2})^n+\frac{1}{2}(1-\sqrt{2})^n} ~=~ \textstyle{\sum_{k\geq 0}{n\choose 2k}{\,}2^k.}
$$
\end{proposition}

\begin{proof}
We clearly have $r(0)=r(1)=1$. Now let $n\geq 2$. By Lemma~\ref{lemma:leqn} the number $r(n)$ is nothing other than the number of ordinal sums of projections on $(X_n,\leq_n)$. If $F$ is such an ordinal sum, then the restriction of $F$ to $X'=X_n\setminus C_k$ is an ordinal sum of projections on $(X',{\leq_n}|_{X'})$. Since there are two possible projections on $C_k$ whenever $n_k\geq 2$, it follows that the sequence $(r(n))_{n\geq 0}$ must satisfy the recurrence equation
$$
r(n) ~=~ r(n-1)+2{\,}\sum_{i=2}^nr(n-i),\qquad n\geq 2.
$$
From this recurrence equation, we immediately derive the claimed one. The rest of the proposition follows straightforwardly.
\end{proof}

\begin{remark}
We observe that an alternative expression for $r(n)$ is given by
$$
r(n) ~=~ \sum_{k=1}^n ~\sum_{\textstyle{n_1,\ldots,n_k\geq 1\atop n_1+\cdots +n_k=n}}\prod_{\textstyle{i=1\atop n_i\geq 2}}^k 2{\,},\qquad n\geq 1.
$$
Indeed, by Lemma~\ref{lemma:leqn} the number $r(n)$ is precisely the number of ordinal sums of projections on $\leq_n$. To compute this number, we need to consider all the unordered partitions of $X_n$ and count twice each subset containing at least two elements (because the two projections are to be considered for each such set). Actually, the product provides the exact number of orbits in $\mathcal{F}_n$ corresponding to the signature $(n_1,\ldots,n_k)$.
\end{remark}

Figure~\ref{fig:20g} provides the contour plots of the $q(3)=20$ operations of $\mathcal{F}_3$, when $X_3$ is endowed with $\leq_3$. These operations are organized in a $7\times 6$ array. Those in the first column consist of the $r(3)=7$ ordinal sums of projections on $\leq_3$. Each of the rows represents an orbit and contains all the possible conjugates of the leftmost operation (we omit the duplicates). In turn, the orbits are grouped into $s(3)=4$ different signatures. Also, all these 20 operations are grouped into $p(3)=13$ weak orderings (represented by rounded boxes).

\setlength{\unitlength}{1.9ex}
\begin{figure}[tbp]
\begin{center}
\begin{footnotesize}
\begin{tabular}{cccccc|cc}
$({123\atop 123_{\mathstrut}})$ & $({123\atop 132})$ & $({123\atop 213})$ & $({123\atop 231})$ & $({123\atop 312})$ & $({123\atop 321})$ & $\mathrm{sgn}(F)$ & $|F^{-1}|$\\
\hline
\ovalbox{
\begin{minipage}{3\unitlength}
\begin{center}
\begin{picture}(3,3)
\multiput(0.5,0.5)(0,1){3}{\multiput(0,0)(1,0){3}{\circle*{0.3}}}
\drawline[1](0.5,0.5)(0.5,2.5)\drawline[1](1.5,0.5)(1.5,2.5)\drawline[1](2.5,0.5)(2.5,2.5)
\end{picture}
\\
\begin{picture}(3,3)
\multiput(0.5,0.5)(0,1){3}{\multiput(0,0)(1,0){3}{\circle*{0.3}}}
\drawline[1](0.5,2.5)(2.5,2.5)\drawline[1](0.5,1.5)(2.5,1.5)\drawline[1](0.5,0.5)(2.5,0.5)
\end{picture}
\\
\begin{tiny}$1{\sim}2{\sim}3$\end{tiny}
\end{center}
\end{minipage}
}
&
&
&
&
&
& $(3)$ &
$(3,3,3)$\\
\ovalbox{
\begin{minipage}{3\unitlength}
\begin{center}
\begin{picture}(3,3)
\multiput(0.5,0.5)(0,1){3}{\multiput(0,0)(1,0){3}{\circle*{0.3}}}
\drawline[1](0.5,2.5)(2.5,2.5)(2.5,0.5)\drawline[1](0.5,1.5)(0.5,0.5)\drawline[1](1.5,1.5)(1.5,0.5)
\end{picture}
\\
\begin{picture}(3,3)
\multiput(0.5,0.5)(0,1){3}{\multiput(0,0)(1,0){3}{\circle*{0.3}}}
\drawline[1](0.5,2.5)(2.5,2.5)(2.5,0.5)\drawline[1](0.5,1.5)(1.5,1.5)\drawline[1](0.5,0.5)(1.5,0.5)
\end{picture}
\\
\begin{tiny}$1{\sim}2{\prec}3$\end{tiny}
\end{center}
\end{minipage}
}
&
\ovalbox{
\begin{minipage}{3\unitlength}
\begin{center}
\begin{picture}(3,3)
\multiput(0.5,0.5)(0,1){3}{\multiput(0,0)(1,0){3}{\circle*{0.3}}}
\drawline[1](0.5,1.5)(2.5,1.5)\drawline[1](1.5,0.5)(1.5,2.5)
\put(2.5,1.5){\oval(0.6,2)[r]}\put(0.5,1.5){\oval(0.6,2)[l]}
\end{picture}
\\
\begin{picture}(3,3)
\multiput(0.5,0.5)(0,1){3}{\multiput(0,0)(1,0){3}{\circle*{0.3}}}
\drawline[1](0.5,1.5)(2.5,1.5)\drawline[1](1.5,0.5)(1.5,2.5)
\put(1.5,0.5){\oval(2,0.6)[b]}\put(1.5,2.5){\oval(2,0.6)[t]}
\end{picture}
\\
\begin{tiny}$1{\sim}3{\prec}2$\end{tiny}
\end{center}
\end{minipage}
}
&
&
\ovalbox{
\begin{minipage}{3\unitlength}
\begin{center}
\begin{picture}(3,3)
\multiput(0.5,0.5)(0,1){3}{\multiput(0,0)(1,0){3}{\circle*{0.3}}}
\drawline[1](0.5,2.5)(0.5,0.5)(2.5,0.5)\drawline[1](1.5,2.5)(1.5,1.5)\drawline[1](2.5,2.5)(2.5,1.5)
\end{picture}
\\
\begin{picture}(3,3)
\multiput(0.5,0.5)(0,1){3}{\multiput(0,0)(1,0){3}{\circle*{0.3}}}
\drawline[1](0.5,2.5)(0.5,0.5)(2.5,0.5)\drawline[1](1.5,2.5)(2.5,2.5)\drawline[1](1.5,1.5)(2.5,1.5)
\end{picture}
\\
\begin{tiny}$2{\sim}3{\prec}1$\end{tiny}
\end{center}
\end{minipage}
}
&
&
& $(2,1)$
&
$(2,2,5)$\\
\ovalbox{
\begin{minipage}{3\unitlength}
\begin{center}
\begin{picture}(3,3)
\multiput(0.5,0.5)(0,1){3}{\multiput(0,0)(1,0){3}{\circle*{0.3}}}
\drawline[1](2.5,2.5)(2.5,0.5)\drawline[1](1.5,2.5)(1.5,0.5)\drawline[1](0.5,1.5)(1.5,1.5)
\put(1.5,2.5){\oval(2,0.6)[t]}
\end{picture}
\\
\begin{picture}(3,3)
\multiput(0.5,0.5)(0,1){3}{\multiput(0,0)(1,0){3}{\circle*{0.3}}}
\drawline[1](0.5,2.5)(2.5,2.5)\drawline[1](0.5,1.5)(2.5,1.5)\drawline[1](1.5,1.5)(1.5,0.5)
\put(2.5,1.5){\oval(0.6,2)[r]}
\end{picture}
\\
\begin{tiny}$1{\prec}2{\sim}3$\end{tiny}
\end{center}
\end{minipage}
}
&
&
\ovalbox{
\begin{minipage}{3\unitlength}
\begin{center}
\begin{picture}(3,3)
\multiput(0.5,0.5)(0,1){3}{\multiput(0,0)(1,0){3}{\circle*{0.3}}}
\drawline[1](0.5,2.5)(0.5,0.5)(1.5,0.5)\drawline[1](1.5,2.5)(2.5,2.5)(2.5,0.5)
\end{picture}
\\
\begin{picture}(3,3)
\multiput(0.5,0.5)(0,1){3}{\multiput(0,0)(1,0){3}{\circle*{0.3}}}
\drawline[1](0.5,2.5)(2.5,2.5)(2.5,1.5)\drawline[1](0.5,1.5)(0.5,0.5)(2.5,0.5)
\end{picture}
\\
\begin{tiny}$2{\prec}1{\sim}3$\end{tiny}
\end{center}
\end{minipage}
}
&
&
\ovalbox{
\begin{minipage}{3\unitlength}
\begin{center}
\begin{picture}(3,3)
\multiput(0.5,0.5)(0,1){3}{\multiput(0,0)(1,0){3}{\circle*{0.3}}}
\drawline[1](0.5,2.5)(0.5,0.5)\drawline[1](1.5,2.5)(1.5,0.5)\drawline[1](1.5,1.5)(2.5,1.5)
\put(1.5,0.5){\oval(2,0.6)[b]}
\end{picture}
\\
\begin{picture}(3,3)
\multiput(0.5,0.5)(0,1){3}{\multiput(0,0)(1,0){3}{\circle*{0.3}}}
\drawline[1](0.5,0.5)(2.5,0.5)\drawline[1](0.5,1.5)(2.5,1.5)\drawline[1](1.5,1.5)(1.5,2.5)
\put(0.5,1.5){\oval(0.6,2)[l]}
\end{picture}
\\
\begin{tiny}$3{\prec}1{\sim}2$\end{tiny}
\end{center}
\end{minipage}
}
&
& $(1,2)$ &
$(1,4,4)$\\
\ovalbox{
\begin{minipage}{3\unitlength}
\begin{center}
\begin{picture}(3,3)
\multiput(0.5,0.5)(0,1){3}{\multiput(0,0)(1,0){3}{\circle*{0.3}}}
\drawline[1](0.5,2.5)(2.5,2.5)(2.5,0.5)\drawline[1](0.5,1.5)(1.5,1.5)(1.5,0.5)
\end{picture}
\\
\begin{tiny}$1{\prec}2{\prec}3$\end{tiny}
\end{center}
\end{minipage}
}
&
\ovalbox{
\begin{minipage}{3\unitlength}
\begin{center}
\begin{picture}(3,3)
\multiput(0.5,0.5)(0,1){3}{\multiput(0,0)(1,0){3}{\circle*{0.3}}}
\drawline[1](0.5,1.5)(2.5,1.5)\drawline[1](1.5,0.5)(1.5,2.5)
\put(2.5,1.5){\oval(0.6,2)[r]}\put(1.5,2.5){\oval(2,0.6)[t]}
\end{picture}
\\
\begin{tiny}$1{\prec}3{\prec}2$\end{tiny}
\end{center}
\end{minipage}
}
&
\ovalbox{
\begin{minipage}{3\unitlength}
\begin{center}
\begin{picture}(3,3)
\multiput(0.5,0.5)(0,1){3}{\multiput(0,0)(1,0){3}{\circle*{0.3}}}
\drawline[1](0.5,2.5)(2.5,2.5)(2.5,0.5)\drawline[1](0.5,1.5)(0.5,0.5)(1.5,0.5)
\end{picture}
\\
\begin{tiny}$2{\prec}1{\prec}3$\end{tiny}
\end{center}
\end{minipage}
}
&
\ovalbox{
\begin{minipage}{3\unitlength}
\begin{center}
\begin{picture}(3,3)
\multiput(0.5,0.5)(0,1){3}{\multiput(0,0)(1,0){3}{\circle*{0.3}}}
\drawline[1](0.5,2.5)(0.5,0.5)(2.5,0.5)\drawline[1](1.5,2.5)(2.5,2.5)(2.5,1.5)
\end{picture}
\\
\begin{tiny}$2{\prec}3{\prec}1$\end{tiny}
\end{center}
\end{minipage}
}
&
\ovalbox{
\begin{minipage}{3\unitlength}
\begin{center}
\begin{picture}(3,3)
\multiput(0.5,0.5)(0,1){3}{\multiput(0,0)(1,0){3}{\circle*{0.3}}}
\drawline[1](0.5,1.5)(2.5,1.5)\drawline[1](1.5,0.5)(1.5,2.5)
\put(0.5,1.5){\oval(0.6,2)[l]}\put(1.5,0.5){\oval(2,0.6)[b]}
\end{picture}
\\
\begin{tiny}$3{\prec}1{\prec}2$\end{tiny}
\end{center}
\end{minipage}
}
&
\ovalbox{
\begin{minipage}{3\unitlength}
\begin{center}
\begin{picture}(3,3)
\multiput(0.5,0.5)(0,1){3}{\multiput(0,0)(1,0){3}{\circle*{0.3}}}
\drawline[1](0.5,2.5)(0.5,0.5)(2.5,0.5)\drawline[1](1.5,2.5)(1.5,1.5)(2.5,1.5)
\end{picture}
\\
\begin{tiny}$3{\prec}2{\prec}1$\end{tiny}
\end{center}
\end{minipage}
}
& $(1,1,1)$ &
$(1,3,5)$\\
\end{tabular}
\end{footnotesize}
\caption{Classifications of the 20 associative and quasitrivial operations on $(X_3,\leq_3)$}
\label{fig:20g}
\end{center}
\end{figure}

Proposition~\ref{prop:s7adf5} can also be easily illustrated in Figure~\ref{fig:20g} as follows. Any permutation $\sigma\in\mathfrak{S}_3$ that maps $F$ to $F_{\sigma}$ can be extended to a permutation of the corresponding rounded boxes (within the same signature).

We end this section by a discussion on the concept of preimage sequence. We know from Proposition~\ref{prop:444} that the preimage sequence of any operation $F\in\mathcal{F}_n$ contains the same information as its signature. Also, it has been shown \cite[Prop.~2.2]{CouDevMar2} that
\begin{equation}\label{eq:sd54}
|F^{-1}[x]| ~=~ 2\times |\{z\in X_n:z\prec_F x\}|+|\{z\in X_n:z\sim_F x\}|,\qquad x\in X_n.
\end{equation}
From the latter identity we can actually derive the following formula:
\begin{equation}\label{eq:sd541}
|F^{-1}| ~=~ ({\,}\underbrace{n_1}_{n_1}{\,},{\,}\underbrace{2{\,}n_1+n_2}_{n_2}{\,},\ldots,{\,}\underbrace{\textstyle{2\sum_{i<k}n_i+n_k}}_{n_k}{\,}).
\end{equation}
Conversely, the signature $\mathrm{sgn}(F)=(n_1,\ldots,n_k)$ can be obtained immediately by considering the absolute frequencies of the sequence $|F^{-1}|$. That is, if $d_1,\ldots,d_k$ represent the distinct values of the sequence $|F^{-1}|$ in increasing order, then $n_i$ is the number of times $d_i$ occurs in $|F^{-1}|$.\footnote{In particular, we observe that, for any $F\in\mathcal{F}_n$, the number of distinct values in the sequence $|F^{-1}|$ is exactly $|X_n/{\sim_F}|$.} To give an example, let $F\in\mathcal{F}_9$ be such that $\mathrm{sgn}(F)=(1,2,2,1,3)$. Then
$$
|F^{-1}| ~=~ (1,4,4,8,8,11,15,15,15).
$$

The following proposition solves the natural question of finding necessary and sufficient conditions for a nondecreasing $n$-sequence $(c_1,\ldots,c_n)$ to be the preimage sequence of an operation $F\in\mathcal{F}_n$.

\begin{proposition}\label{prop:s4d3}
Let $c=(c_1,\ldots,c_n)$ be a nondecreasing $n$-sequence. Then there exists $F\in\mathcal{F}_n$ such that $|F^{-1}|=c$ if and only if
\begin{equation}\label{eq:s4d3}
c_i ~=~ \min\{j:c_j=c_i\}+\max\{j:c_j=c_i\}-1,\qquad i=1,\ldots,n.
\end{equation}
\end{proposition}

\begin{proof}
(Necessity) Replacing $F$ with one of its conjugates if necessary, we can assume that $|F^{-1}[1]|\leq\cdots\leq |F^{-1}[n]|$. By \eqref{eq:of1}, we then have $1\precsim\cdots\precsim n$. For every $i\in\{1,\ldots,n\}$, define
\begin{eqnarray*}
p_i &=& \min\{j:|F^{-1}[j]|=|F^{-1}[i]|\},\\
q_i &=& \max\{j:|F^{-1}[j]|=|F^{-1}[i]|\}.
\end{eqnarray*}
By \eqref{eq:sd54}, we then have $|F^{-1}[i]|=2(p_i-1)+(q_i-p_i+1)=p_i+q_i-1$.

(Sufficiency) Let $c=(c_1,\ldots,c_n)$ be a nondecreasing $n$-sequence satisfying the stated condition and let $n_1,\ldots,n_k$ be the absolute frequencies of this sequence. Take any $F\in\mathcal{F}_n$ such that $\mathrm{sgn}(F)=(n_1,\ldots,n_k)$. By definition, for any $\ell\in\{1,\ldots,k\}$, all the components of $c$ corresponding to frequency $n_{\ell}$ are equal to the number
$$
\textstyle{(\sum_{i<\ell} n_i+1)+(\sum_{i<\ell} n_i + n_{\ell})-1 ~=~ 2\sum_{i<\ell} n_i+n_{\ell}.}
$$
Equation~\eqref{eq:sd541} then shows that $|F^{-1}|=c$.

Let us provide an alternative proof that does not make use of \eqref{eq:sd541}. We proceed by induction on $n$. The result clearly holds for $n=1$. Suppose that it holds for any $\ell\leq n-1$ and let us prove that it still holds for $n$.

Let $c=(c_1,\ldots,c_n)$ be a nondecreasing $n$-sequence satisfying the stated condition. If $c_1=c_n$, then we can take $F=\pi_1$ or $F=\pi_2$ on $X_n$. If $c_1<c_n$, then let $\ell =\max\{j:c_j<c_n\}$. By the induction hypothesis, there exists $F_{\ell}\in\mathcal{F}_{\ell}$ such that $|F_{\ell}^{-1}|=(c_1,\ldots,c_{\ell})$. Now, let $F\colon X^2_n\to X_n$ be defined by
$$
F(x,y) ~=~
\begin{cases}
F_{\ell}(x,y), & \text{if $x,y\in X_{\ell}$},\\
\pi_1(x,y), & \text{if $x,y\in X_n\setminus X_{\ell}$},\\
\max(x,y), & \text{otherwise}.
\end{cases}
$$
Then it is not difficult to see that $F\in\mathcal{F}_n$ and that $|F^{-1}|=c$.
\end{proof}

\begin{remark}
There are quasitrivial operations $F\colon X_n^2\to X_n$ that are not associative and whose preimage sequences $|F^{-1}|$ satisfy condition \eqref{eq:s4d3}. The operation $F\colon X_3^2\to X_3$ whose contour plot is shown in Figure~\ref{fig:378} could serve as an example here.
\end{remark}

\setlength{\unitlength}{3.5ex}
\begin{figure}[htbp]
\begin{center}
\begin{picture}(3,3)
\multiput(0.5,0.5)(0,1){3}{\multiput(0,0)(1,0){3}{\circle*{0.2}}}
\drawline[1](0.5,0.5)(1.5,0.5)\drawline[1](0.5,1.5)(1.5,1.5)(1.5,2.5)\drawline[1](2.5,0.5)(2.5,2.5)
\put(0.5,1.5){\oval(0.6,2)[l]}
\end{picture}
\caption{A quasitrivial operation $F\colon X_3^2\to X_3$ that is not associative}
\label{fig:378}
\end{center}
\end{figure}
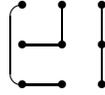

\section{Order-preserving operations}

Recall that an operation $F\colon X^2\to X$ is said to be \emph{$\leq$-preserving} for some total ordering $\leq$ on $X$ if for any $x,y,x',y'\in X$ such that $x\leq x'$ and $y\leq y'$, we have $F(x,y)\leq F(x',y')$.

\begin{definition}
We say that an operation $F\colon X^2\to X$ is \emph{order-preservable} if there exists a total ordering $\leq$ on $X$ for which $F$ is $\leq$-preserving.
\end{definition}

In this section we provide characterizations of the operations $F\in\mathcal{F}$ that are order-preservable (see Proposition~\ref{prop:s76dfv} below). To this extent, we first recall a characterization of the operations $F\in\mathcal{F}$ that are $\leq$-preserving for some given total ordering $\leq$ on $X$. Recall also that a weak ordering $\precsim$ on $X$ is said to be \emph{single-plateaued} (or \emph{weakly single-peaked}) \cite[Def.~4.3]{CouDevMar2} for some total ordering $\leq$ on $X$ if for any $a,b,c\in X$ such that $a<b<c$, we have $b\prec a$ or $b\prec c$ or $a\sim b\sim c$.

\begin{proposition}[{see \cite[Prop.~4.4]{CouDevMar2}}]\label{prop:sf7dw}
Let $\leq$ be a total ordering on $X$ and let $F\in\mathcal{F}$. Then $F$ is $\leq$-preserving if and only if $\precsim_F$ is single-plateaued for $\leq$.
\end{proposition}

We say that a weak ordering $\precsim$ on $X$ is \emph{$2$-quasilinear} if there are no pairwise distinct $a,b,c,d\in X$ such that $a\prec b\sim c\sim d$. That is, a weak ordering $\precsim$ on $X$ is $2$-quasilinear if and only if every set $C\in X/{\sim}$ that is not minimal for $\precsim$ contains at most two elements of $X$.

\begin{proposition}\label{prop:mainFnF}
Assume the axiom of choice. A weak ordering on $X$ is $2$-quasilinear if and only if it is single-plateaued for some total ordering on $X$.
\end{proposition}

\begin{proof}
(Necessity) Let ${\precsim}$ be a $2$-quasilinear weak ordering on $X$. For any $x\in X$, let $S_{x/{\sim}}$ be a total ordering on $x/{\sim}=\{z\in X:z\sim x\}$. Such total orderings always exist by the axiom of choice.

Consider the binary relation $\leq$ on $X$ whose symmetric part is the identity relation on $X$ and the asymmetric part $<$ is defined as follows. Let $x,y\in X$ be such that $x\neq y$ and $x\precsim y$.
\begin{eqnarray*}
\text{$-~$ If $~x\sim y$} & \text{and} &
\hspace{-0.5ex}
\begin{cases}
~\text{$x{\,}S_{x/{\sim}}{\,}y$}{\,}, & \text{then we set $x<y$.}\\
~\text{$y{\,}S_{x/{\sim}}{\,}x$}{\,}, & \text{then we set $y<x$.}
\end{cases}
\\
\text{$-~$ If $~x\prec y$} & \text{and} &
\hspace{-0.5ex}
\begin{cases}
~\text{$y=\min_{S_{y/{\sim}}}y/{\sim}$}{\,}, & \text{then we set $y<x$.}\\
~\text{$y=\max_{S_{y/{\sim}}}y/{\sim}$ and $|y/{\sim}|=2$}{\,}, & \text{then we set $x<y$.}
\end{cases}
\end{eqnarray*}

Let us show that, thus defined, the relation $\leq$ is a total ordering on $X$. It is clearly total by definition. It is also antisymmetric. Indeed, it is clear that there are no $x,y\in X$ such that $x<y$ and $y<x$. Finally, let us prove by contradiction that it is transitive. Suppose that there are $x,y,z\in X$ such that $x<y$, $y<z$, and $z<x$. Also, suppose for instance that $x\sim y\prec z$ (the other $p(3)-1=12$ cases can be verified similarly). Since $y\prec z$, we must have $z=\max_{S_{z/{\sim}}}z/{\sim}$ and $|z/{\sim}|=2$. Also, since $x\prec z$, we must have $z=\min_{S_{z/{\sim}}}z/{\sim}$, a contradiction.

Let us now show that $\precsim$ is single-plateaued for $\leq$. Let $a,b,c\in X$ such that $a<b<c$, $\neg(a\sim b\sim c)$, and $c\precsim b$. We only need to show that $b\prec a$. We have two exclusive cases to consider.
\begin{itemize}
\item If $c\sim b$, then we clearly have $b{\,}S_{b/{\sim}}{\,}c$. It follows that we cannot have $a\prec b$ and $b=\mathrm{max}_{S_{b/{\sim}}}b/{\sim}$ and $|b/{\sim}|=2$. We cannot have $a\sim b$ either for we have $\neg(a\sim b\sim c)$. Therefore, we must have $b\prec a$.
\item If $c\prec b$, then we have $b=\mathrm{min}_{S_{b/{\sim}}}b/{\sim}$. It follows that we cannot have $a\prec b$ and $b=\mathrm{max}_{S_{b/{\sim}}}b/{\sim}$ and $|b/{\sim}|=2$. Clearly, we cannot have $a\sim b$ and $a{\,}S_{b/{\sim}}{\,}b$ either. Therefore, we must have $b\prec a$.
\end{itemize}
(Sufficiency) This was proved in \cite[Lemma~4.6]{CouDevMar2}.
\end{proof}

\begin{remark}
When $X=X_n$ for some integer $n\geq 1$, the total ordering $\leq$ on $X_n$ mentioned in Proposition~\ref{prop:mainFnF} can be very easily constructed as follows. First, choose a total ordering $S_i$ on each set $C_i$ ($i=1,\ldots,k$). Then, execute the following four-step algorithm.
\begin{itemize}
\item[1.] Let $L$ be the empty list.
\item[2.] For $i=k,\ldots,2$, append the element $\min_{S_i}C_i$ to $L$.
\item[3.] Append to $L$ the elements of $C_1$ in the order given by $S_1$.
\item[4.] For $i=2,\ldots,k$ such that $|C_i|=2$, append the element $\max_{S_i}C_i$ to $L$.
\end{itemize}
The order given by $L$ defines a suitable total ordering $\leq$ on $X_n$. For instance, suppose that we have the following $2$-quasilinear weak ordering $\precsim$ on $X_8$:
$$
1\sim 2\sim 3\prec 4\sim 5\prec 6\prec 7\sim 8.
$$
In each set $C_i$ ($i=1,2,3,4$), we let $S_i$ be the restriction of $\leq_n$ to $C_i$. The list constructed by the algorithm above is then given by $L=(7,6,4,1,2,3,5,8)$ and provides the following total ordering:
$$
7<6<4<1<2<3<5<8.
$$
Finally, we can readily verify that $\precsim$ is single-plateaued for $\leq$.
\end{remark}

\begin{proposition}\label{prop:s76dfv}
Assume the axiom of choice. For any $F\in\mathcal{F}$, the following assertions are equivalent.
\begin{itemize}
\item[(i)] $F$ is order-preservable.
\item[(ii)] There exists $\sigma\in\mathfrak{S}$ such that $F_{\sigma}$ is order-preservable.
\item[(iii)] $F_{\sigma}$ is order-preservable for every $\sigma\in\mathfrak{S}$.
\item[(iv)] ${\precsim_F}$ is $2$-quasilinear.
\end{itemize}
If $X=X_n$ for some integer $n\geq 1$, then any of the assertions above is equivalent to any of the following ones.
\begin{itemize}
\item[(v)] There exists $\sigma\in\mathfrak{S}_n$ such that $F_{\sigma}$ is $\leq_n$-preserving.
\item[(vi)] $n_2,\ldots,n_k\in\{1,2\}$, where $(n_1,\ldots,n_k)=\mathrm{sgn}(F)$.
\item[(vii)] Every integer strictly greater than $c_1$ occurs at most two times in $|F^{-1}|=(c_1,\ldots,c_n)$.
\end{itemize}
Moreover, the equivalence among (i), (ii), (iii), and (v) holds for any operation $F\colon X^2\to X$.
\end{proposition}

\begin{proof}
(i) $\Leftrightarrow$ (ii) $\Leftrightarrow$ (iii). These equivalences are straightforward.

(i) $\Leftrightarrow$ (iv). This follows from both Propositions~\ref{prop:sf7dw} and \ref{prop:mainFnF}.

(i) $\Rightarrow$ (v). Let $\leq$ be a total ordering $X_n$ for which $F$ is $\leq$-preserving. Take $\sigma\in\mathfrak{S}_n$ such that
$$
\sigma(x)\leq_n\sigma(y) \quad\Leftrightarrow\quad x\leq y,\qquad x,y\in X_n.
$$
Now let $x,x',y,y'\in X_n$ such that $x\leq_n x'$ and $y\leq_n y'$. We then have $\sigma^{-1}(x)\leq\sigma^{-1}(x')$ and $\sigma^{-1}(y)\leq\sigma^{-1}(y')$. Since $F$ is $\leq$-preserving, we have
$$
F(\sigma^{-1}(x),\sigma^{-1}(y)) ~\leq ~ F(\sigma^{-1}(x'),\sigma^{-1}(y')),
$$
that is, $F_{\sigma}(x,y)\leq_n F_{\sigma}(x',y')$.

(v) $\Rightarrow$ (ii). Trivial.

(iv) $\Leftrightarrow$ (vi). Trivial.

(iv) $\Leftrightarrow$ (vii). This follows from \eqref{eq:of1}.
\end{proof}

\begin{corollary}
Let $c=(c_1,\ldots,c_n)$ be a nondecreasing $n$-sequence. Then there exists an order-preservable operation $F\in\mathcal{F}_n$ such that $|F^{-1}|=c$ if and only if both Eq.~\eqref{eq:s4d3} and assertion (vii) of Proposition~\ref{prop:s76dfv} hold.
\end{corollary}

\begin{proof}
This result immediately follows from Propositions~\ref{prop:s4d3} and \ref{prop:s76dfv}.
\end{proof}

For any $F,G\in\mathcal{F}$ such that $\mathrm{sgn}(F)=\mathrm{sgn}(G)$, by Proposition~\ref{prop:444} we have ${\precsim_F}\simeq{\precsim_G}$, and hence ${\precsim_F}$ is $2$-quasilinear if and only if so is ${\precsim_G}$. By Proposition~\ref{prop:s76dfv}, it follows that $F$ is order-preservable if and only if so is $G$. In particular, for any $\sigma\in\mathfrak{S}$, we have that $F$ is order-preservable if and only if so is $F_{\sigma}$, as also mentioned in Proposition~\ref{prop:s76dfv}. This observation justifies the following terminology. For any order-preservable operation $F\in\mathcal{F}$, we say that its signature $\mathrm{sgn}(F)$ and orbit $\mathrm{orb}(F)$ are \emph{order-preservable}.

Let us assume for the rest of this section that $X=X_n$ for some integer $n\geq 1$. It is not difficult to see by inspection that all the signatures in $\mathcal{F}_n$ are order-preservable when $n\leq 3$. For $n=4$, only the signature $(1,3)$ is not order-preservable. It consists of two non-order-preservable orbits and corresponds to the preimage sequence $(1,5,5,5)$. Figure~\ref{fig:nop} (left) shows the contour plot of one of the eight non-order-preservable operations in $\mathcal{F}_4$.

\setlength{\unitlength}{3.5ex}
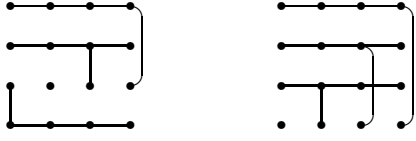
\begin{figure}[htbp]
\begin{center}
\begin{small}
\begin{picture}(4,4)
\multiput(0.5,0.5)(0,1){4}{\multiput(0,0)(1,0){4}{\circle*{0.2}}}
\drawline[1](0.5,3.5)(3.5,3.5)\drawline[1](0.5,2.5)(3.5,2.5)\drawline[1](2.5,2.5)(2.5,1.5)
\drawline[1](0.5,1.5)(0.5,0.5)(3.5,0.5)\put(3.5,2.5){\oval(0.6,2)[r]}
\end{picture}
\hspace{0.1\textwidth}
\begin{picture}(4,4)
\multiput(0.5,0.5)(0,1){4}{\multiput(0,0)(1,0){4}{\circle*{0.2}}}
\drawline[1](0.5,3.5)(3.5,3.5)\drawline[1](0.5,2.5)(3.5,2.5)\drawline[1](1.5,0.5)(1.5,1.5)
\drawline[1](0.5,1.5)(3.5,1.5)\put(2.5,1.5){\oval(0.6,2)[r]}\put(3.5,2){\oval(0.6,3)[r]}
\end{picture}
\end{small}
\caption{A non-order-preservable operation in $\mathcal{F}_4$ (left) and its ordinal sum representation (right)}
\label{fig:nop}
\end{center}
\end{figure}

Let us now consider enumeration problems. We first observe that the number of operations $F\in\mathcal{F}_n$ that are $\leq_n$-preserving was computed in \cite[Prop.~4.11]{CouDevMar2} and is known as Sloane's $\mathrm{A293005}$.

Now, for any integer $n\geq 0$, let $p_{\mathrm{op}}(n)$ be the number of $2$-quasilinear weak orderings on $X_n$ and let $q_{\mathrm{op}}(n)$ be the number of order-preservable operations $F\in\mathcal{F}_n$. Let also $r_{\mathrm{op}}(n)$ be the number of order-preservable orbits in $\mathcal{F}_n$ and let $s_{\mathrm{op}}(n)$ be the number of order-preservable signatures in $\mathcal{F}_n$. By convention, we set $p_{\mathrm{op}}(0)=q_{\mathrm{op}}(0)=r_{\mathrm{op}}(0)=s_{\mathrm{op}}(0)=1$. The next three propositions provide explicit expressions for these sequences. Also, the first few values are given in Table~\ref{tab:pquv}.

\begin{table}[htbp]
$$
\begin{array}{|c|rrrr|}
\hline n & p_{\mathrm{op}}(n) & q_{\mathrm{op}}(n) & r_{\mathrm{op}}(n) & s_{\mathrm{op}}(n) \\
\hline 1 & 1 & 1 & 1 & 1 \\
2 & 3 & 4 & 3 & 2 \\
3 & 13 & 20 & 7 & 4 \\
4 & 71 & 130 & 15 & 7 \\
5 & 486 & 1{\,}052 & 31 & 12 \\
6 & 3{\,}982 & 10{\,}214 & 63 & 20 \\
\hline
\end{array}
$$
\caption{First few values of $p_{\mathrm{op}}(n)$, $q_{\mathrm{op}}(n)$, $r_{\mathrm{op}}(n)$, and $s_{\mathrm{op}}(n)$}
\label{tab:pquv}
\end{table}

\begin{proposition}\label{prop:prn}
The sequence $(p_{\mathrm{op}}(n))_{n\geq 0}$ satisfies the second order linear recurrence equation
$$
p_{\mathrm{op}}(n+2) ~=~ 1+(n+2){\,}p_{\mathrm{op}}(n+1)+\frac{1}{2}(n+2)(n+1){\,}p_{\mathrm{op}}(n),\qquad n\geq 1,
$$
with $p_{\mathrm{op}}(1)=1$ and $p_{\mathrm{op}}(2)=3$, and we have
$$
p_{\mathrm{op}}(n) ~=~ \sum_{k=0}^n\frac{n!}{(n+1-k)!}{\,}G_k{\,},\qquad n\geq 1,
$$
where $G_n=\frac{\sqrt{3}}{3}\,(\frac{1+\sqrt{3}}{2})^n-\frac{\sqrt{3}}{3}\,(\frac{1-\sqrt{3}}{2})^n$. Moreover, its exponential generating function is $\widehat{P}_{\mathrm{op}}(z)=(2e^z-2z-z^2)/(2-2z-z^2)$.
\end{proposition}

\begin{proof}
We clearly have $p_{\mathrm{op}}(1)=1$ and $p_{\mathrm{op}}(2)=3$. Now, let $n\geq 3$, let ${\precsim}$ be a $2$-quasilinear weak ordering on $X_n$, and let $m$ be the number of maximal elements of $X_n$ for $\precsim$. By definition of $2$-quasilinearity, we necessarily have $m\in\{1,2,n\}$. Moreover, the restriction of $\precsim$ to the $(n-m)$-element set obtained from $X_n$ by removing its maximal elements for $\precsim$ is $2$-quasilinear. It follows that the sequence $p_{\mathrm{op}}(n)$ satisfies the second order linear recurrence equation
$$
p_{\mathrm{op}}(n) ~=~ 1+n{\,}p_{\mathrm{op}}(n-1)+\frac{1}{2}{\,}n(n-1){\,}p_{\mathrm{op}}(n-2),\qquad n\geq 3,
$$
as claimed. Thus, the sequence $(a(n))_{n\geq 0}$ defined by $a(n)=p_{\mathrm{op}}(n)/n!$ for every $n\geq 0$ satisfies the second order linear recurrence equation (with constant coefficients)
$$
a(n) ~=~ \frac{1}{n!}+a(n-1)+\frac{1}{2}{\,}a(n-2),\qquad n\geq 3.
$$
The expression for the exponential generating function of $(p_{\mathrm{op}}(n))_{n\geq 0}$ (which is exactly the ordinary generating function of $(a(n))_{n\geq 0}$) follows straightforwardly. The claimed closed form for $p_{\mathrm{op}}(n)$ is then obtained by solving the latter recurrence equation (using the method of variation of parameters).
\end{proof}

\begin{proposition}\label{prop:FibLuc}
The sequence $(q_{\mathrm{op}}(n))_{n\geq 0}$ satisfies the second order linear recurrence equation
$$
q_{\mathrm{op}}(n+2) ~=~ 2+(n+2){\,}q_{\mathrm{op}}(n+1)+(n+2)(n+1){\,}q_{\mathrm{op}}(n),\qquad n\geq 1,
$$
with $q_{\mathrm{op}}(1)=1$ and $q_{\mathrm{op}}(2)=4$, and we have
$$
q_{\mathrm{op}}(n) ~=~ n!{\,}F_n+2\,\sum_{k=0}^{n-1}\frac{n!}{(n+1-k)!}{\,}F_k{\,},\qquad n\geq 1,
$$
where $F_n=\frac{\sqrt{5}}{5}\,(\frac{1+\sqrt{5}}{2})^n-\frac{\sqrt{5}}{5}\,(\frac{1-\sqrt{5}}{2})^n$ is the $n$th Fibonacci number. Moreover, its exponential generating function is $\widehat{Q}_{\mathrm{op}}(z)=(2e^z-1-2z-z^2)/(1-z-z^2)$.
\end{proposition}

\begin{proof}
The proof is similar to that of Proposition~\ref{prop:prn}.
\end{proof}

\begin{proposition}
We have $r_{\mathrm{op}}(n)=2^n-1$ and $s_{\mathrm{op}}(n)=F_{n+2}-1$ for every $n\geq 1$, where $F_n$ is the $n$th Fibonacci number.
\end{proposition}

\begin{proof}
We clearly have $r_{\mathrm{op}}(1)=1$ and $r_{\mathrm{op}}(2)=3$. To compute $r_{\mathrm{op}}(n)$ for $n\geq 3$, we proceed exactly as in the proof of Proposition~\ref{prop:d6fssf}, except that here we have $n_k\in\{1,2,n\}$. It follows that the sequence $r_{\mathrm{op}}(n)$ satisfies the second order linear recurrence equation
$$
r_{\mathrm{op}}(n) ~=~ 2+r_{\mathrm{op}}(n-1)+2{\,}r_{\mathrm{op}}(n-2),\qquad n\geq 3.
$$
The explicit expression for $r_{\mathrm{op}}(n)$ then follows immediately.

Let us now consider the sequence $s_{\mathrm{op}}(n)$. We clearly have $s_{\mathrm{op}}(1)=1$ and $s_{\mathrm{op}}(2)=2$. Let $n\geq 3$. We know by Proposition~\ref{prop:444} that $s_{\mathrm{op}}(n)$ is also the number of $2$-quasilinear weak orderings on $X_n$ that are defined up to an isomorphism. Thus, proceeding as in the proof of Proposition~\ref{prop:prn}, we see that the sequence $s_{\mathrm{op}}(n)$ satisfies the second order linear recurrence equation
$$
s_{\mathrm{op}}(n) ~=~ 1+s_{\mathrm{op}}(n-1)+s_{\mathrm{op}}(n-2),\qquad n\geq 3.
$$
The explicit expression for $s_{\mathrm{op}}(n)$ then follows immediately.
\end{proof}

\section{Commutative, anticommutative, and bisymmetric operations}

Recall that an operation $F\colon X^2\to X$ is said to be
\begin{itemize}
\item \emph{commutative} if $F(x,y)=F(y,x)$ for all $x,y\in X$;
\item \emph{anticommutative} if, for any $x,y\in X$, the equality $F(x,y)=F(y,x)$ implies $x=y$;
\item \emph{bisymmetric} (or \emph{medial}) if $F(F(x,y),F(u,v))=F(F(x,u),F(y,v))$ for all $x,y,u,v\in X$.
\end{itemize}

In this final section, we investigate the subclasses of $\mathcal{F}$ defined by each of the three properties above. As far as commutative or anticommutative operations are concerned, we have the following two propositions.

\begin{proposition}
Let $F\colon X^2\to X$ be an operation. The following assertions are equivalent.
\begin{enumerate}
\item[(i)] $F\in\mathcal{F}$ and is commutative.
\item[(ii)] $F$ is quasitrivial, order-preservable, and commutative.
\item[(iii)] $F=\max_{\preceq}$ for some total ordering $\preceq$ on $X$.
\end{enumerate}
If $X=X_n$ for some integer $n\geq 1$, then any of the assertions (i)--(iii) above is equivalent to any of the following ones.
\begin{enumerate}
\item[(iv)] $F\in\mathcal{F}_n$ and $|\mathrm{orb}(F)|=n!$.
\item[(v)] $F\in\mathcal{F}_n$ and $\mathrm{sgn}(F)=(1,\ldots,1)$.
\item[(vi)] $F$ is quasitrivial and satisfies $|F^{-1}|=(1,3,5,\ldots,2n-1)$.
\item[(vii)] $F$ is associative, idempotent (i.e., $F(x,x)=x$ for all $x\in X_n$), order-preservable, commutative, and has a neutral element.
\end{enumerate}
\end{proposition}

\begin{proof}
The equivalence among (i), (ii), (iii), (vi), and (vii) was proved in \cite[Theorem~3.3]{CouDevMar2}. The equivalence (iv) $\Leftrightarrow$ (v) immediately follows from Proposition~\ref{prop:staborb}. The equivalence (iii) $\Leftrightarrow$ (v) is trivial.
\end{proof}

\begin{proposition}
Let $F\colon X^2\to X$ be an operation. The following assertions are equivalent.
\begin{enumerate}
\item[(i)] $F\in\mathcal{F}$ and is anticommutative.
\item[(ii)] $F$ is quasitrivial, order-preservable, and anticommutative.
\item[(iii)] $F=\pi_1$ or $F=\pi_2$.
\end{enumerate}
If $X=X_n$ for some integer $n\geq 1$, then any of the assertions (i)--(iii) above is equivalent to any of the following ones.
\begin{enumerate}
\item[(iv)] $F\in\mathcal{F}_n$ and $|\mathrm{orb}(F)|=1$.
\item[(v)] $F\in\mathcal{F}_n$ and $\mathrm{sgn}(F)=(n)$.
\item[(vi)] $F$ is quasitrivial, order-preservable, and satisfies $|F^{-1}|=(n,\ldots,n)$.
\end{enumerate}
\end{proposition}

\begin{proof}
(i) $\Leftrightarrow$ (iii). This follows from Theorem~\ref{thm:kimura}.

(iii) $\Rightarrow$ ((ii) \& (vi)). Trivial.

(ii) $\Rightarrow$ (iii). Let $\leq$ be a total ordering on $X$ for which $F$ is $\leq$-preserving. Let $x,y\in X$ such that $x<y$. Since $F$ is quasitrivial and anticommutative, it follows that $F|_{\{x,y\}^2}=\pi_1|_{\{x,y\}^2}$ or $F|_{\{x,y\}^2}=\pi_2|_{\{x,y\}^2}$. Suppose that $F|_{\{x,y\}^2}=\pi_1|_{\{x,y\}^2}$ (the other case is similar). Let $z\in X\setminus\{x,y\}$ and let us show that $F|_{\{x,z\}^2}=\pi_1|_{\{x,z\}^2}$. We then have the following discussion of cases.
\begin{itemize}
\item If $x<z<y$, then $x=F(x,x)\leq F(x,z)\leq F(x,y)=x$. We then have $F(x,z)=x$ and hence $F(z,x)=z$.
\item If $y<z$, then $y=F(y,x)\leq F(z,x)\in\{x,z\}$. We then have $F(z,x)=z$ and hence $F(x,z)=x$.
\item The case $z<x$ is similar to the previous one.
\end{itemize}
Therefore, we have $F|_{\{x,z\}^2}=\pi_1|_{\{x,z\}^2}$. Similarly, we can show that $F|_{\{u,v\}^2}=\pi_1|_{\{u,v\}^2}$ for any $u,v\in X$.

(vi) $\Rightarrow$ (iii). By Proposition~\ref{prop:s76dfv}, there exists $\sigma\in\mathfrak{S}_n$ such that $F_{\sigma}$ is $\leq_n$-preserving. Clearly, $F_{\sigma}$ is quasitrivial. Also, by Proposition~\ref{prop:FG1C} (using $\sigma$ to define the graph isomorphism) we have that $|F_{\sigma}^{-1}|=(n,\ldots,n)$. Now, it was proved in \cite[Prop.~3.11]{Dev} that condition (iii) is equivalent to saying that $F$ is quasitrivial, $\leq_n$-preserving, and satisfies $|F^{-1}|=(n,\ldots,n)$. This completes the proof.

(iii) $\Leftrightarrow$ (iv) $\Leftrightarrow$ (v). This follows from Proposition~\ref{prop:staborb}.
\end{proof}

Let us now investigate those operations in $\mathcal{F}$ that are bisymmetric. Recall first that any bisymmetric and quasitrivial operation is also associative and hence it is an element of $\mathcal{F}$ (see, e.g., \cite[Cor.\ 10.3]{Kep81}).

A weak ordering $\precsim$ on $X$ is said to be \emph{quasilinear} if there are no pairwise distinct $a,b,c\in X$ such that $a\prec b\sim c$ (see \cite[Def.\ 3.1]{Dev}). That is, a weak ordering $\precsim$ on $X$ is quasilinear if and only if every set $C\in X/{\sim}$ that is not minimal for $\precsim$ contains exactly one element of $X$. Clearly, such a weak ordering is also $2$-quasilinear. We also have the following proposition, which is the counterpart of Proposition~\ref{prop:mainFnF} for quasilinear weak orderings.

\begin{proposition}\label{prop:qlsp5}
A weak ordering on $X$ is quasilinear if and only if it is single-plateaued for any total ordering on $X$ that extends it.
\end{proposition}

\begin{proof}
(Necessity) Let $\precsim$ be a quasilinear weak ordering on $X$. Suppose that there exists a total ordering $\leq$ on $X$ that extends $\precsim$ and such that $\precsim$ is not single-plateaued for $\leq$. That is, there exist $a,b,c\in X$ such that $a<b<c$, $c\precsim b$, $a\precsim b$, and $\neg(a\sim b\sim c)$. Then we must have $a\prec b\sim c$, which contradicts quasilinearity.

(Sufficiency) Let $\precsim$ be a weak ordering on $X$ that is single-plateaued for any total ordering on $X$ that extends it. Suppose that $\precsim$ is not quasilinear; that is, there exist pairwise distinct $a,b,c\in X$ such that $a\prec b\sim c$. It is then clear that $\precsim$ is not single-plateaued for any total ordering $\leq$ on $X$ that extends $\precsim$ and such that $a<b<c$. Thus, we reach a contradiction.
\end{proof}

We now have the following proposition.

\begin{proposition}\label{prop:bi}
Let $F\colon X^2\to X$ be an operation. The following assertions are equivalent.
\begin{enumerate}
\item[(i)] $F$ is bisymmetric and quasitrivial.
\item[(ii)] $F\in\mathcal{F}$ and $\precsim_F$ is quasilinear.
\item[(iii)] $F\in\mathcal{F}$ and is $\leq$-preserving for every total ordering $\leq$ on $X$ that extends $\precsim_F$.
\end{enumerate}
If $X=X_n$ for some integer $n\geq 1$, then any of the assertions (i)--(iii) above is equivalent to any of the following ones.
\begin{enumerate}
\item[(iv)] $F\in\mathcal{F}_n$ and and there exists $\ell\in\{1,\ldots,n\}$ such that
    $$
    \mathrm{sgn}(F)~=~(\ell,\underbrace{1,\ldots,1}_{n-\ell}).
    $$
\item[(v)] $F\in\mathcal{F}_n$ and there exists $\ell\in\{1,\ldots,n\}$ such that
    $$
    |F^{-1}|~=~(\underbrace{\ell,\ldots,\ell}_{\ell},2\ell +1,2\ell +3,\ldots,2n-1).
    $$
\item[(vi)] $F$ is quasitrivial, order-preservable, and there exists $\ell\in\{1,\ldots,n\}$ such that
    $$
    |F^{-1}|~=~(\underbrace{\ell,\ldots,\ell}_{\ell},2\ell +1,2\ell +3,\ldots,2n-1).
    $$
\end{enumerate}
\end{proposition}

\begin{proof}
The equivalence among (i), (ii), (iii), and (v) was proved in \cite[Theorem~3.6]{Dev}. The equivalence (ii) $\Leftrightarrow$ (iv) and the implication ((iii) \& (v)) $\Rightarrow$ (vi) are trivial.

Let us now prove that (vi) $\Rightarrow$ (i). By Proposition~\ref{prop:s76dfv}, there exists $\sigma\in\mathfrak{S}_n$ such that $F_{\sigma}$ is $\leq_n$-preserving. Clearly, $F_{\sigma}$ is quasitrivial. Also, by Proposition~\ref{prop:FG1C} (using $\sigma$ to define the graph isomorphism) we have that
$$
|F_{\sigma}^{-1}|~=~(\underbrace{\ell,\ldots,\ell}_{\ell},2\ell +1,2\ell +3,\ldots,2n-1).
$$
Now, using \cite[Theorem 3.14]{Dev} it follows that $F_{\sigma}$ is bisymmetric and quasitrivial, and hence so is $F$.
\end{proof}

\begin{remark}
We observe that Proposition~\ref{prop:qlsp5} can also be easily established by using Theorem~\ref{thm:kimura} and Propositions~\ref{prop:sf7dw} and \ref{prop:bi}.
\end{remark}

Let us now consider enumeration problems. We first observe that the number of quasilinear weak orderings on $X_n$ was computed in \cite[Prop.\ 4.1]{Dev} and is known as Sloane's A002627. Also, the number of bisymmetric operations in $\mathcal{F}_n$ was computed in \cite[Prop.\ 4.2]{Dev} and is known as Sloane's A296943.

For any $F,G\in\mathcal{F}$ such that $\mathrm{sgn}(F)=\mathrm{sgn}(G)$, by Proposition~\ref{prop:bi} we have that $F$ is bisymmetric if and only if so is $G$. Thus, for any bisymmetric operation in $\mathcal{F}$, we can say that its signature and orbit are \emph{bisymmetric}.

For any integer $n\geq 0$, let $r_{\mathrm{b}}(n)$ be the number of bisymmetric orbits in $\mathcal{F}_n$ and let $s_{\mathrm{b}}(n)$ be the number of bisymmetric signatures in $\mathcal{F}_n$. By convention, we set $r_{\mathrm{b}}(0)=s_{\mathrm{b}}(0)=1$.

\begin{proposition}
We have $r_{\mathrm{b}}(n)=2n-1$ and $s_{\mathrm{b}}(n)=n$ for any $n\geq 1$.
\end{proposition}

\begin{proof}
We clearly have $r_{\mathrm{b}}(1)=1$. To compute $r_{\mathrm{b}}(n)$ for $n\geq 2$, we proceed exactly
as in the proof of Proposition~\ref{prop:d6fssf}, except that here we have $n_k\in\{1,n\}$. It follows
that the sequence $r_{\mathrm{b}}(n)$ satisfies the first order linear recurrence equation
$$
r_{\mathrm{b}}(n) = 2 + r_{\mathrm{b}}(n-1),\qquad n\geq 2.
$$
The explicit expression for $r_{\mathrm{b}}(n)$ then follows immediately.

Let us now consider the sequence $s_{\mathrm{b}}(n)$. We clearly have $s_{\mathrm{b}}(1)=1$. Let $n\geq 2$.
We know by Proposition~\ref{prop:444} that $s_{\mathrm{b}}(n)$ is also the number of quasilinear weak
orderings on $X_n$ that are defined up to an isomorphism. Thus, proceeding as in the proof of Proposition~\ref{prop:prn}, we see that the sequence $s_{\mathrm{b}}(n)$ satisfies the first order linear recurrence equation
$$
s_{\mathrm{b}}(n) = 1 + s_{\mathrm{b}}(n-1),\qquad n\geq 2.
$$
The explicit expression for $s_{\mathrm{b}}(n)$ then follows immediately.
\end{proof}

\section*{Acknowledgments}

This research is supported by the Internal Research Project R-AGR-0500 of the University of Luxembourg and the Luxembourg National Research Fund R-AGR-3080.

\end{document}